\documentclass[final,1p,number,fleqn]{elsarticle}

\usepackage{amsmath,amsthm,amssymb}

\makeatletter
\newcommand{\subjclassname@later}{\textup{2010} Mathematics Subject Classification}
\makeatother

\newtheorem{theorem}{Theorem}[section]

\newtheorem{corollary}{Corollary}[section]

\theoremstyle{definition}
\newtheorem{remark}{Remark}[section]
\numberwithin{equation}{section} \numberwithin{theorem}{section}

\DeclareMathOperator{\RE}{Re}

\journal{Applied Mathematics Letters} 

\begin{document}
\begin{frontmatter}

\title{Coefficient estimates for bi-univalent Ma-Minda starlike\\
 and convex functions\tnoteref{t1}}

\tnotetext[t1]{The work presented here was supported in part by RU
and FRGS research grants from Universiti Sains Malaysia, and
University of Delhi.   The authors are thankful to the referee
 for the comments.}

\author[rma]{Rosihan M. Ali\corref{cor1}}
\address[rma]{School of Mathematical Sciences,
Universiti Sains Malaysia, 11800 USM, Penang, Malaysia}
\ead{rosihan@cs.usm.my}

\author[rma]{See Keong Lee}  \ead{sklee@cs.usm.my }

\author[rma,vr]{V. Ravichandran}
\address[vr]{Department of Mathematics, University of Delhi,
Delhi--110 007, India} \ead{vravi@maths.du.ac.in}

\author[rma]{Shamani Supramaniam}
\ead{sham105@hotmail.com}

\cortext[cor1]{Corresponding author}

\begin{abstract}Estimates on the initial
coefficients are obtained for normalized analytic  functions $f$ in the open unit disk with $f$ and its inverse $g=f^{-1}$ satisfying the
conditions that $zf'(z)/f(z)$ and $zg'(z)/g(z)$ are both subordinate to a starlike univalent function whose range is symmetric with respect to
the real axis. Several related classes of functions are also considered, and connections to earlier known results are made.
\end{abstract}


\end{frontmatter}

\section{Introduction} Let $\mathcal{A}$ be the class of all analytic
functions $f$ in the open unit disk $\mathbb{D} : = \{ z \in
\mathbb{C} : |z| < 1 \}$ and normalized by the conditions $f(0)= 0$
and $f'(0)= 1$. The Koebe one-quarter theorem \cite{duren}  ensures
that the image of $\mathbb{D}$ under every univalent function
$f\in\mathcal{A}$ contains a disk of radius 1/4. Thus every
univalent function $f$ has an inverse $f^{- 1}$ satisfying $f^{-
1}(f(z))= z$, $(z \in \mathbb{D})$ and
\[ f(f^{- 1}(w))= w, \quad  \left(|w| < r_0(f), r_0(f)
   \geq 1 / 4 \right). \]
A function $f \in \mathcal{A}$ is said to be bi-univalent in $\mathbb{D}$ if both $f $ and $f^{- 1}$ are univalent in $\mathbb{D}$. Let $\sigma$
denote the class of bi-univalent functions defined in the unit disk $\mathbb{D}$. A domain $D \subset \mathbb{C}$ is \textit{convex} if the line
segment joining any two points in $D$ lies entirely in $D,$ while a domain is \textit{starlike} with respect to a point $w_0 \in D$ if the line
segment joining any point of $D$ to $w_0$ lies inside $D$. A function $f \in \mathcal{A}$ is starlike if $f( \mathbb{D})$ is a starlike domain
with respect to the origin, and convex if $f (\mathbb{D})$ is convex. Analytically, $f \in \mathcal{A}$ is starlike if and only if
$\text{\textrm{Re}}\ zf'(z)/ f(z)> 0$, whereas $f \in \mathcal{A}$ is convex if and only if $1 + \text{\textrm{Re}}\ zf''(z )/ f'(z)> 0$. The
classes consisting of starlike and convex functions are denoted by $\mathcal{ST}$ and $\mathcal{CV}$ respectively. The classes
$\mathcal{ST}(\alpha)$ and $\mathcal{CV}(\alpha)$ of starlike and convex functions of order $\alpha$, $0 \leq \alpha < 1$, are respectively
characterized by $\text{\textrm{Re}}\ zf'(z)/ f(z)> \alpha$ and $1 + \text{\textrm{Re}} \
 zf''(z)/ f'(z)> \alpha$. Various subclasses of
starlike and convex functions are often investigated. These functions are typically characterized by the quantity $zf'(z)/ f (z)$ or $1 +
zf''(z)/ f'(z)$ lying in a certain domain starlike with respect to 1 in the right-half plane. Subordination is useful to unify these subclasses.

An analytic function $f$ is subordinate to an analytic function $g$, written $f(z)\prec g(z)$, provided there is an analytic function $w$
defined on $\mathbb{D}$ with $w(0)= 0$ and $|w(z)| < 1$ satisfying $f(z)= g(w(z))$. Ma and Minda {\cite{mamin}} unified various subclasses of
starlike and convex functions for which either of the quantity $zf'(z)/ f(z)$ or $1 + zf''(z)/ f'(z)$ is subordinate to a more general
superordinate function. For this purpose, they considered an analytic function $\varphi$ with positive real part in the unit disk $\mathbb{D}$,
$\varphi(0)= 1$, $\varphi'(0)> 0$, and $\varphi$ maps $\mathbb{D}$ onto a region starlike with respect to $1$ and symmetric with respect to the
real axis. The class of Ma-Minda starlike functions consists of functions $f \in \mathcal{A}$ satisfying the subordination $zf'(z)/ f(z)\prec
\varphi( z)$. Similarly, the class of Ma-Minda convex functions consists of functions $f \in \mathcal{A}$ satisfying the subordination $1 +
zf''(z)/ f'(z)\prec \varphi(z)$. A function $f$ is bi-starlike of Ma-Minda type or bi-convex of Ma-Minda type if both $f$ and $f^{- 1}$ are
respectively Ma-Minda starlike or convex. These classes are denoted respectively by $\mathcal{ST}_{\sigma}(\varphi)$ and
$\mathcal{CV}_{\sigma}(\varphi)$.

Lewin {\cite{lewin}} investigated the class $\sigma$ of bi-univalent functions and obtained the bound for the second coefficient. Several
authors have subsequently studied similar problems in this direction (see {\cite{netan69,brankir}}). Brannan and Taha {\cite{brantaha}}
considered certain subclasses of bi-univalent functions, similar to the familiar subclasses of univalent functions consisting of strongly
starlike, starlike and convex functions. They introduced bi-starlike functions and bi-convex functions and obtained estimates on the initial
coefficients. Recently, Srivastava \emph{et al.}\ {\cite{sriv}} introduced and investigated subclasses of bi-univalent functions and obtained
bounds for the initial coefficients. Bounds for the initial coefficients of several classes of functions were also investigated in
\cite{rma1,rma2,vravi,pon1,pon2}.

In this paper, estimates on the initial coefficients for bi-starlike
of Ma-Minda type and bi-convex of Ma-Minda type functions are
obtained. Several related classes are also considered, and
connection to earlier known results are made. The classes introduced
in this paper are motivated by the corresponding classes
investigated  in \cite{rma2}.

\section{Coefficient Estimates}
In the sequel, it is assumed that $\varphi$ is an analytic function
with positive real part in the unit disk $\mathbb{D}$, with
$\varphi(0)= 1$, $\varphi'(0)> 0$,  and  $\varphi(\mathbb{D})$ is
symmetric   with respect to the real axis. Such a function has a
series expansion of the form
\begin{equation}
  \label{varphi1} \varphi(z)= 1 + B_1 z + B_2 z^2 + B_3 z^3 + \cdots , \quad (B_1>0).
\end{equation}
A function $f\in\mathcal{A}$ with $\RE(f'(z))>0$ is known to be
univalent. This motivates the following class of functions. A
function $f \in \sigma$ is said to be in
  the class $\mathcal{H}_{\sigma}(\varphi)$ if the following subordinations
  hold:
  \[ f'(z)\prec \varphi(z)\quad  \text{and} \quad  g'(w
)\prec \varphi(w), \quad g(w): = f^{- 1}(w).\] For functions in the class $\mathcal{H}_{\sigma}(\varphi)$, the following result is obtained.

\begin{theorem}
  If $f \in \mathcal{H}_{\sigma}(\varphi)$ is given by
  \begin{equation}
  \label{eqf} f(z)= z + \sum_{n = 2}^{\infty} a_n z^n ,
\end{equation}
 then
  \begin{equation}
    \label{re1} |a_2 | \leq \frac{B_1\sqrt{B_1}}{\sqrt{|3 B_1^2 - 4 B_2 + 4 B_1|}}
    \quad  \text{and} \quad  |a_3 | \leq  \left(\frac{1}{3} +
    \frac{B_1}{4} \right)B_1.
  \end{equation}
\end{theorem}

\begin{proof}
  Let $f \in \mathcal{H}_{\sigma}(\varphi)$ and $g=f^{-1}$. Then there are analytic functions
  $u, v : \mathbb{D} \rightarrow \mathbb{D}$, with $u(0)= v(0)= 0$,
  satisfying
  \begin{equation}
    \label{eq1} f'(z)= \varphi(u(z))\quad  \text{and}
    \quad  g'(w)= \varphi(v(w)).
  \end{equation}
  Define the functions $p_1$ and $p_2$ by
  \begin{equation*}
    \label{eqpu} p_1(z): = \frac{1 + u(z)}{1 - u(z)}=1 + c_1 z + c_2 z^2 +
     \cdots \quad  \text{and} \quad
  p_2(z): = \frac{1 + v(z )}{1 - v(z)}= 1+ b_1 z + b_2 z^2 + \cdots, \end{equation*} or, equivalently,
  \begin{equation}
    \label{equ} u(z)= \frac{p_1(z)- 1}{p_1(z)+ 1}=\frac{1}{2} \left( c_1 z
     + \left(c_2 - \frac{c_1^2}{2} \right)z^2 + \cdots \right)\end{equation} and
     \begin{equation}\label{eqv} v(z)= \frac{p_2(z)- 1}{p_2(z)+ 1}=\frac{1}{2} \left( b_1
     z     + \left(b_2 - \frac{b_1^2}{2} \right)z^2 + \cdots \right) .\end{equation}
Then $p_1$ and $p_2$ are analytic in $\mathbb{D}$ with $p_1(0)=1=p_2(0)$. Since $u, v : \mathbb{D} \rightarrow \mathbb{D}$, the functions $p_1$
and $p_2$ have positive real part in $\mathbb{D}$, and $|b_i|\leq 2 $ and $|c_i|\leq 2$.
  In view of {\eqref{eq1}},  \eqref{equ} and \eqref{eqv}, clearly
  \begin{equation}
    \label{eq2} f'(z)= \varphi \left(\frac{p_1(z)- 1}{p_1(z)+ 1}
    \right)\quad  \text{and} \quad  g'(w)= \varphi \left(
    \frac{p_2(w)- 1}{p_2(w)+ 1} \right).
  \end{equation}
Using \eqref{equ} and \eqref{eqv} together with {\eqref{varphi1}},
it is evident that
  \begin{equation}
    \label{var1} \varphi \left(\frac{p_1(z)- 1}{p_1(z)+ 1} \right)=
    1 + \frac{1}{2} B_1 c_1 z + \left( \frac{1}{2} B_1 \left(c_2 -
    \frac{c_1^2}{2} \right)+ \frac{1}{4} B_2 c_1^2 \right) z^2 + \cdots
  \end{equation}
  and
  \begin{equation}
    \label{var2} \varphi \left(\frac{p_2(w)- 1}{p_2(w)+ 1} \right)=
    1 + \frac{1}{2} B_1 b_1 w + \left( \frac{1}{2} B_1 \left(b_2 -
    \frac{b_1^2}{2} \right)+ \frac{1}{4} B_2 b_1^2 \right) w^2 + \cdots .
  \end{equation}
Since $f \in \mathcal{\sigma}$ has the  Maclaurin series given by
\eqref{eqf}, a computation shows that its inverse $g=f^{- 1}$ has
the expansion
\[ g(w)=f^{- 1}(w)= w -a_2 w^2 +(2 a_2^2 - a_3)w^3 + \cdots . \]
  Since
  \[ f'(z)= 1 + 2 a_2 z + 3 a_3 z^2 + \cdots \quad  \text{and}
     \quad  g'(w)= 1 - 2 a_2 w + 3(2 a_2^2 - a_3)w^2 + \cdots,
  \]
  it follows from {\eqref{eq2}}, {\eqref{var1}} and {\eqref{var2}} that
  \begin{equation}
    \label{eq1.1} 2 a_2 = \frac{1}{2} B_1 c_1,
  \end{equation}
  \begin{equation}
    \label{eq1.2} 3 a_3 = \frac{1}{2} B_1 \left(c_2 - \frac{c_1^2}{2} \right)
    + \frac{1}{4} B_2 c_1^2,
  \end{equation}
  \begin{equation}
    \label{eq1.3} - 2 a_2 = \frac{1}{2} B_1 b_1
  \end{equation}
  and
  \begin{equation}
    \label{eq1.4} 3(2 a_2^2 - a_3)= \frac{1}{2} B_1 \left(b_2 -
    \frac{b_1^2}{2} \right)+ \frac{1}{4} B_2 b_1^2 .
  \end{equation}
  From {\eqref{eq1.1}} and {\eqref{eq1.3}}, it follows that
  \begin{equation}
    \label{eq1.5} c_1 = - b_1.
  \end{equation}
  Now {\eqref{eq1.2}}, {\eqref{eq1.3}}, {\eqref{eq1.4}} and {\eqref{eq1.5}}
 yield
  \[ a_2^2 = \frac{B_1^3(b_2 + c_2)}{4(3 B_1^2 - 4 B_2 + 4 B_1)}, \]
  which, in view of the well-known inequalities $|b_2 | \leq 2$ and
  $|c_2 | \leq 2$ for functions
  with positive real part,  gives us the desired estimate on $|a_2 |$ as asserted in
  {\eqref{re1}}.

  By subtracting  {\eqref{eq1.4}} from {\eqref{eq1.2}}, further computations using
  \eqref{eq1.1} and \eqref{eq1.5} lead to
  \[ a_3 = \frac{1}{12} B_1(c_2 - b_2)+ \frac{1}{16} B_1^2 c_1^2, \]
  and this yields the estimate given in {\eqref{re1}}.
\end{proof}

\begin{remark}
  For the class of strongly starlike functions, the function $\varphi$ is given by
  \[ \varphi(z)= \left(\frac{1 + z}{1 - z} \right)^{\gamma} = 1 + 2
     \gamma z + 2 \gamma^2 z^2 + \cdots \quad (0<\gamma\leq 1), \]
  which gives $B_1 = 2 \gamma$ and $B_2 = 2 \gamma^2$. Hence the inequalities
  in {\eqref{re1}} reduce to the result in \cite[Theorem 1, inequality (2.4), p.3]{sriv}.
  In the case
  \[ \varphi(z)= \frac{1 +(1 - 2 \gamma)z}{1 - z} = 1 + 2(1 - \gamma
)z + 2(1 - \gamma)z^2 + \cdots, \]
  then $B_1 = B_2 = 2(1 - \gamma)$, and thus the inequalities in
  {\eqref{re1}} reduce to the result in \cite[Theorem 2, inequality (3.3), p.4]{sriv}.
\end{remark}

A function $f \in \sigma$ is said to be in  the class
$\mathcal{ST}_{\sigma}(\alpha, \varphi)$, $\alpha\geq 0$, if the
following subordinations  hold:
  \[ \frac{zf'(z)}{f(z)}+\frac{\alpha z^{2}f''(z)}{f(z)} \prec \varphi(z)\quad
  \text{and} \quad  \frac{wg'(w)}{g(w)} + \frac{\alpha w^{2} g''(w)}{g(w)}
  \prec \varphi(w), \quad g(w): = f^{- 1}(w).\]  Note that
$\mathcal{ST}_{\sigma}(\varphi)\equiv
\mathcal{ST}_{\sigma}(0,\varphi)$. For functions in the class
$\mathcal{ST}_{\sigma}(\alpha, \varphi)$, the following coefficient
estimates are obtained.

\begin{theorem}\label{the2}
  Let $f$ given by \eqref{eqf} be in the  class $\mathcal{ST}_{\sigma}
(\alpha, \varphi)$. Then
  \begin{equation}
    \label{new1} |a_2 | \leq \frac{B_1\sqrt{B_1}}{\sqrt{|B_1^2(1+4\alpha) + (B_1 -
    B_2)(1+2\alpha)^2}|},
  \end{equation}
  and
  \begin{equation}
    \label{new2} |a_3 | \leq \frac{B_1 + |B_2 - B_1 |}{(1+4\alpha)}.
  \end{equation}
\end{theorem}

\begin{proof}
Let $f\in \mathcal{ST}_{\sigma}(\alpha, \varphi)$.
  Then there are analytic
  functions $u, v : \mathbb{D} \rightarrow \mathbb{D}$, with $u(0)= v(0
)= 0$, satisfying
  \begin{equation}
    \label{newth1} \frac{zf'(z)}{f(z)}+\frac{\alpha z^{2}f''(z)}{f(z)} = \varphi(u(z))\quad
    \text{and} \quad  \frac{wg'(w)}{g(w)} + \frac{\alpha w^{2} g''(w)}{g(w)} = \varphi(v(w)), \quad (g=f^{-1}).
  \end{equation}
  Since
  \[ \frac{zf'(z)}{f(z)}+\frac{\alpha z^{2}f''(z)}{f(z)} = 1 + a_2(1+2\alpha) z
  +(2(1+3\alpha) a_3 -(1+2\alpha) a_2^2)z^2 + \cdots \]
  and
  \[  \frac{wg'(w)}{g(w)} + \frac{\alpha w^{2} g''(w)}{g(w)}
  = 1 -(1+2\alpha) a_2 w +((3+10\alpha) a_2^2 - 2(1+3\alpha) a_3)w^2 + \cdots,
  \]
then {\eqref{var1}}, {\eqref{var2}} and {\eqref{newth1}} yield
  \begin{equation}
    \label{eq2.1n} a_2(1+2\alpha) = \frac{1}{2} B_1 c_1,
  \end{equation}
  \begin{equation}
    \label{eq2.2n} 2(1+3\alpha) a_3 -(1+2\alpha) a_2^2 = \frac{1}{2} B_1 \left(c_2 - \frac{c_1^2}{2}
    \right)+ \frac{1}{4} B_2 c_1^2,
  \end{equation}
  \begin{equation}
    \label{eq2.3n} - (1+2\alpha)a_2 = \frac{1}{2} B_1 b_1
  \end{equation}
  and
  \begin{equation}
    \label{eq2.4n} (3+10\alpha) a_2^2 - 2(1+3\alpha) a_3 = \frac{1}{2} B_1 \left(b_2 -
    \frac{b_1^2}{2} \right)+ \frac{1}{4} B_2 b_1^2 .
  \end{equation}
 It follows from {\eqref{eq2.1n}} and {\eqref{eq2.3n}} that
  \begin{equation}
    \label{eq2.5n} c_1 = - b_1.
  \end{equation}
  Equations {\eqref{eq2.2n}}, {\eqref{eq2.3n}}, {\eqref{eq2.4n}} and
  {\eqref{eq2.5n}} lead to
  \[ a_2^2 = \frac{B_1^3(b_2 + c_2)}{4(B_1^2(1+4\alpha) + (B_1 - B_2)(1+2\alpha)^2)}, \]
  which,  in view of the  inequalities $|b_2 | \leq 2$ and
  $|c_2 | \leq 2$ for functions
  with positive real part,  yield
\[ |a_2|^2 \leq  \frac{B_1^3}{|B_1^2(1+4\alpha) + (B_1 - B_2)(1+2\alpha)^2|}. \]
Since $B_1>0$, the last inequality, upon taking square roots, gives
the desired estimate on $|a_2 |$ given in  {\eqref{new1}}.

Now, further computations from {\eqref{eq2.2n}}, {\eqref{eq2.3n}},
{\eqref{eq2.4n}} and {\eqref{eq2.5n}} lead to
\[ a_3 = \frac{({B_1}/2)((3+10\alpha) c_2 + (1+2\alpha)b_2)+  b_1^2 (1+3\alpha)(B_2 -
  B_1)}{4(1+3\alpha)(1+4\alpha)}, \] which, using  the  inequalities $|b_1|\leq 2$, $|b_2 | \leq 2$ and
$|c_2 | \leq 2$ for functions  with positive real part,   yields
\[
|a_3| \leq  \frac{({B_1}/2)(2(3+10\alpha) + 2(1+2\alpha))+ 4
(1+3\alpha)(B_2 -
  B_1)}{4(1+3\alpha)(1+4\alpha)}
=  \frac{B_1 + |B_2 - B_1 |}{(1+4\alpha)}.\] This completes the
proof of the estimate in  {\eqref{new2}}.
\end{proof}

For  $\alpha=0$, Theorem~\ref{the2} readily yields the following coefficient estimates
for Ma-Minda bi-starlike functions.

\begin{corollary}
  Let $f$ given by \eqref{eqf} be in the class $\mathcal{ST}_{\sigma}
(\varphi)$. Then
  \begin{equation*}
|a_2 | \leq \frac{B_1\sqrt{B_1}}{\sqrt{|B_1^2 + B_1 -
    B_2|}}\quad \text{ and } \quad
  |a_3 | \leq B_1 + |B_2 - B_1 |.
  \end{equation*}
\end{corollary}

\begin{remark}
  For the class of strongly starlike functions, the function $\varphi$ is given by
  \[ \varphi(z)= \left(\frac{1 + z}{1 - z} \right)^{\gamma} = 1 + 2
     \gamma z + 2 \gamma^2 z^2 + \cdots \quad (0<\gamma\leq 1), \]
 and so $B_1 = 2 \gamma$ and $B_2 = 2 \gamma^2$. Hence, when $\alpha=0$ (bi-starlike function),
 the inequality in   {\eqref{new1}} reduces to the estimates in {\cite[Theorem 2.1]{brantaha}}. On the other hand,
  when $\alpha=0$ and
  \[ \varphi(z)= \frac{1 +(1 - 2 \gamma)z}{1 - z} = 1 + 2(1 - \gamma
)z + 2(1 - \gamma)z^2 + \cdots, \]
  then $B_1 = B_2 = 2(1 - \gamma)$ and thus the inequalities in
  {\eqref{new1}} and {\eqref{new2}} reduce to the estimates in {\cite[Theorem 3.1]{brantaha}}.
\end{remark}

Next, a  function $f \in \sigma$ belongs to the class
$\mathcal{M}_{\sigma}(\alpha, \varphi)$, $\alpha\geq0$, if the
following subordinations  hold:
  \[ (1-\alpha)\frac{zf'(z)}{f(z)}+\alpha\left(1+\frac{ zf''(z)}{f'(z)}\right) \prec \varphi(z)\] and
   \[(1-\alpha) \frac{wg'(w)}{g(w)} + \alpha\left(1+\frac{ w g''(w)}{g'(w)}\right)
  \prec \varphi(w),\] $g(w): = f^{- 1}(w)$. A function in the class $\mathcal{M}_{\sigma}(\alpha, \varphi)$
  is called bi-Mocanu-convex function of Ma-Minda type. This class unifies the
classes $\mathcal{ST}_{\sigma}(\varphi)$ and
$\mathcal{CV}_{\sigma}(\varphi)$.

For functions in the class $\mathcal{M}_{\sigma}(\alpha, \varphi)$,
the following coefficient estimates hold.

\begin{theorem}\label{the3.1}
  Let $f$ given by \eqref{varphi1} be in the class $\mathcal{M}_{\sigma}(\alpha, \varphi)$. Then
  \begin{equation}
    \label{re5n} |a_2 | \leq \frac{B_1\sqrt{B_1}}{\sqrt{(1+\alpha)|B_1^2 + (1+\alpha)( B_1 - B_2)|}}
  \end{equation}
  and
  \begin{equation}
    \label{re6n} |a_3 | \leq \frac{ B_1 + |B_2 - B_1 |}{1+\alpha}.
  \end{equation}
\end{theorem}

\begin{proof}
If  $f\in \mathcal{M}_{\sigma}(\alpha, \varphi)$, then there are
analytic
  functions $u, v : \mathbb{D} \rightarrow \mathbb{D}$, with $u(0)= v(0
)= 0$, such that
  \begin{align}\label{th3na}
     (1-\alpha)\frac{zf'(z)}{f(z)}+\alpha\left(1+\frac{ zf''(z)}{f'(z)}\right)&=
     \varphi(u(z))
    \intertext{and}     (1-\alpha) \frac{wg'(w)}{g(w)} + \alpha\left(1+\frac{ w g''(w)}{g'(w)}\right) &=
    \varphi(v(w)).\label{th3nb}
  \end{align}
  Since
  \[ (1-\alpha)\frac{zf'(z)}{f(z)}+\alpha\left(1+\frac{ zf''(z)}{f'(z)}\right)
  = 1 + (1+\alpha) a_2 z +(2(1+2\alpha) a_3 - (1+3\alpha)a_2^2)z^2 +
     \cdots \]
  and
  \[(1-\alpha) \frac{wg'(w)}{g(w)} + \alpha\left(1+\frac{ w g''(w)}{g'(w)}\right)
   = 1 - (1+\alpha) a_2 w +((3+5\alpha) a_2^2 - 2(1+2\alpha) a_3)w^2 +
     \cdots, \]
  from {\eqref{var1}}, {\eqref{var2}},  {\eqref{th3na}} and {\eqref{th3nb}},
  it follows that
  \begin{equation}
    \label{eq3.1n} (1+\alpha) a_2 = \frac{1}{2} B_1 c_1,
  \end{equation}
  \begin{equation}
    \label{eq3.2n} 2(1+2\alpha) a_3 - (1+3\alpha) a_2^2 = \frac{1}{2} B_1 \left(c_2 -
    \frac{c_1^2}{2} \right)+ \frac{1}{4} B_2 c_1^2,
  \end{equation}
  \begin{equation}
    \label{eq3.3n} - (1+\alpha) a_2 = \frac{1}{2} B_1 b_1
  \end{equation}
  and
  \begin{equation}
    \label{eq3.4n}  (3+5\alpha) a_2^2 - 2(1+2\alpha) a_3 = \frac{1}{2} B_1 \left(b_2 -
    \frac{b_1^2}{2} \right)+ \frac{1}{4} B_2 b_1^2 .
  \end{equation}
 The equations {\eqref{eq3.1n}} and {\eqref{eq3.3n}} yield
  \begin{equation}
    \label{eq3.5n} c_1 = - b_1.
  \end{equation}
  From {\eqref{eq3.2n}}, {\eqref{eq3.4n}} and {\eqref{eq3.5n}}, it
  follows that
  \[ a_2^2 = \frac{B_1^3(b_2 + c_2)}{4(1+\alpha)(B_1^2 + (1+\alpha)( B_1 -  B_2)}, \]
  which yields the desired estimate on $|a_2 |$ as described in
  {\eqref{re5n}}.

As in the earlier proofs, use of {\eqref{eq3.2n}}, \eqref{eq3.3n}, {\eqref{eq3.4n}} and \eqref{eq3.5n}
shows  that
  \[ a_3 =  \frac{({B_1}/2)((1+3\alpha)b_2+(3+5\alpha)c_2)+b_1^2(1+2\alpha)(B_2 -
  B_1)}{4(1+\alpha)(1+2\alpha)},
  \]
which yields the estimate {\eqref{re6n}}.
\end{proof}

For $\alpha=0$, Theorem~\ref{the3.1} gives the coefficient estimates for  Ma-Minda bi-starlike functions, while for $\alpha=1$, it gives the
following estimates for Ma-Minda bi-convex functions.
\begin{corollary}
  Let $f$ given by \eqref{varphi1} be in the class $\mathcal{CV}_{\sigma}
(\varphi)$. Then
  \begin{equation*}
  |a_2 | \leq \frac{B_1\sqrt{B_1}}{\sqrt{2|B_1^2 + 2 B_1 - 2
    B_2|}}\quad \text{and}\quad
    |a_3 | \leq \frac{1}{2}(B_1 + |B_2 - B_1 |).
  \end{equation*}
\end{corollary}

\begin{remark}
  For $\varphi$ given by
  \[ \varphi(z)= \frac{1 +(1 - 2 \gamma)z}{1 - z} = 1 + 2(1 - \gamma
)z + 2(1 - \gamma)z^2 + \cdots, \]
  evidently $B_1 = B_2 = 2(1 - \gamma)$, and thus when $\alpha=1$ (bi-convex functions), the inequalities in
  {\eqref{re5n}} and {\eqref{re6n}} reduce to a result in {\cite[Theorem 4.1]{brantaha}}.
\end{remark}

Next, function $f \in \sigma$ is said to be in  the class
$\mathcal{L}_{\sigma}(\alpha, \varphi)$, $\alpha\geq0$, if the
following subordinations  hold:
  \[ \left(\frac{zf'(z)}{f(z)}\right)^\alpha \left(1+\frac{ zf''(z)}{f'(z)}\right)^{1-\alpha} \prec \varphi(z)\] and
   \[ \left(\frac{wg'(w)}{g(w)}\right)^\alpha \left(1+\frac{ w
   g''(w)}{g'(w)}\right)^{1-\alpha}
  \prec \varphi(w),\] $g(w): = f^{- 1}(w)$.
 This class also reduces to the classes of Ma-Minda
bi-starlike and bi-convex functions. For functions in this class, the following coefficient estimates are obtained.

\begin{theorem}
  Let $f$ given by \eqref{varphi1} be in the class $\mathcal{L}_{\sigma}(\alpha, \varphi)$. Then
\begin{equation}
    \label{re5n2} |a_2 | \leq
  \frac{2B_1\sqrt{B_1}}{\sqrt{|2(\alpha^2-3\alpha+4)
    B_1^2+4(\alpha-2)^2( B_1 - B_2)|}}
\end{equation}
  and
\begin{equation}
    \label{re6n2} |a_3 | \leq \frac{2(3-2\alpha)\Big(B_1 +|B_1 - B_2 |\Big)}
    {|(3-2\alpha)(\alpha^2-3\alpha+4)|}.
  \end{equation}
\end{theorem}

\begin{proof}
  Let $f\in \mathcal{L}_{\sigma}(\alpha, \varphi)$.
Then there are analytic
  functions $u, v : \mathbb{D} \rightarrow \mathbb{D}$, with $u(0)= v(0
)= 0$, such that
  \begin{align}  \label{th3n2a}
     \left(\frac{zf'(z)}{f(z)}\right)^\alpha \left(1+\frac{ zf''(z)}{f'(z)}\right)^{1-\alpha} &=
     \varphi(u(z))
        \intertext{and}    \label{th3n2b}
        \left(\frac{wg'(w)}{g(w)}\right)^\alpha \left(1+\frac{ w
   g''(w)}{g'(w)}\right)^{1-\alpha} &=
    \varphi(v(w)).
  \end{align}
  Since
  \[ \left(\frac{zf'(z)}{f(z)}\right)^\alpha \left(1+\frac{ zf''(z)}{f'(z)}\right)^{1-\alpha} = 1 + (2-\alpha) a_2 z +\left(2(3-2\alpha) a_3
  +
  \frac{(\alpha-2)^2 -3(4-3\alpha)}{2}a_2^2\right)z^2 +
     \cdots \]
  and
  \[\left(\frac{wg'(w)}{g(w)}\right)^\alpha \left(1+\frac{ w
   g''(w)}{g'(w)}\right)^{1-\alpha} = 1 - (2-\alpha) a_2 w +\Big((8(1-\alpha)+\frac{1}{2}\alpha(\alpha+5))a_2^2 - 2(3-2\alpha) a_3\Big)w^2 +
     \cdots, \]
  from {\eqref{var1}}, {\eqref{var2}}, {\eqref{th3n2a}} and {\eqref{th3n2b}},
  it follows that
  \begin{equation}
    \label{eq3.1n2} (2-\alpha) a_2 = \frac{1}{2} B_1 c_1,
  \end{equation}
  \begin{equation}
    \label{eq3.2n2} 2(3-2\alpha) a_3 +\Big((\alpha-2)^2-3(4-3\alpha)\Big) \frac{a_2^2}{2} = \frac{1}{2} B_1 \left(c_2 -
    \frac{c_1^2}{2} \right)+ \frac{1}{4} B_2 c_1^2,
  \end{equation}
  \begin{equation}
    \label{eq3.3n2} - (2-\alpha) a_2 = \frac{1}{2} B_1 b_1
  \end{equation}
  and
  \begin{equation}
    \label{eq3.4n2}  \Big(8(1-\alpha)+\frac{1}{2}\alpha(\alpha+5)\Big)a_2^2 - 2(3-2\alpha) a_3 = \frac{1}{2} B_1 \left(b_2 -
    \frac{b_1^2}{2} \right)+ \frac{1}{4} B_2 b_1^2 .
  \end{equation}
  Now {\eqref{eq3.1n2}} and {\eqref{eq3.3n2}} clearly yield
  \begin{equation}
    \label{eq3.5n2} c_1 = - b_1.
  \end{equation}
  Equations {\eqref{eq3.2n2}}, {\eqref{eq3.4n2}} and {\eqref{eq3.5n2}} lead to
  \[ a_2^2 = \frac{B_1^3(b_2 + c_2)}{2(\alpha^2-3\alpha+4)B_1^2+4(\alpha-2)^2( B_1 - B_2)}, \]
  which yields the desired estimate on $|a_2 |$ as asserted in
  {\eqref{re5n2}}.

  Proceeding similarly as in the earlier proof, using {\eqref{eq3.2n2}}, \eqref{eq3.3n2}, {\eqref{eq3.4n2}} and \eqref{eq3.5n2},
 it follows that
  \[ a_3 =
  \frac{({B_1}/2)\Big((16(1-\alpha)+\alpha(\alpha+5))c_2+(3(4-3\alpha)-(\alpha-2)^2)b_2\Big)
  +2b_1^2(3-2\alpha)(B_1 - B_2)}{4(3-2\alpha)(\alpha^2-3\alpha+4)}
  \]
  which yields the estimate {\eqref{re6n2}}.
\end{proof}

\begin{remark}
The determination of  the sharp estimates for the coefficients
$|a_2|$, $|a_3|$ and for other coefficients of functions belonging
to the classes investigated in this paper  are open problems. In
fact, some estimate (not necessarily sharp) for $|a_n|$, $(n\geq 4)$
would be interesting.
\end{remark}

\end{document}